\documentclass[reqno,a4paper,12pt]{amsart} 

\usepackage{amsmath,amscd,amsfonts,amssymb}
\usepackage{mathrsfs,dsfont}

\numberwithin{equation}{section}
\numberwithin{figure}{section}

\newcommand\R{\mathbb{R}}

\newcommand\Z{\mathbb{Z}}

\newcommand\lam{\lambda}
\newcommand\Lam{\Lambda}

\newcommand\sig{\sigma}

\newcommand\eps{\varepsilon}

\renewcommand\le{\leqslant}

\renewcommand\leq{\leqslant}
\renewcommand\geq{\geqslant}
\newcommand\sbt{\subset}

\renewcommand\hat{\widehat}

\renewcommand\Im{\operatorname{Im}}

\newcommand{\ft}[1]{\widehat #1}
\newcommand{\dotprod}[2]{\langle #1 , #2 \rangle}

\newcommand{\supp}{\operatorname{supp}}

\theoremstyle{plain}
\newtheorem{thm}{Theorem}[section]
\newtheorem{lem}[thm]{Lemma}
\newtheorem{corollary}[thm]{Corollary}

\newtheorem*{claim*}{Claim}

\newcommand{\thmref}[1]{Theorem~\ref{#1}}
\newcommand{\secref}[1]{Section~\ref{#1}}

\newcommand{\lemref}[1]{Lemma~\ref{#1}}

\newcommand{\corref}[1]{Corollary~\ref{#1}}

\newcommand{\remref}[1]{Remark~\ref{#1}}

\theoremstyle{definition}

\newtheorem*{definition*}{Definition}
\newtheorem*{remarks*}{Remarks}
\newtheorem*{remark*}{Remark}
\newtheorem{remark}[thm]{Remark}

\newenvironment{enumerate-roman}
{\begin{enumerate}
\addtolength{\itemsep}{5pt}
}
{\end{enumerate}}

\newenvironment{enumerate-alph}
{\begin{enumerate}
\addtolength{\itemsep}{5pt}
}
{\end{enumerate}}

\newenvironment{enumerate-num}
{\begin{enumerate}
\addtolength{\itemsep}{5pt}
}
{\end{enumerate}}

\newenvironment{enumerate-text}
{\begin{enumerate}
\addtolength{\itemsep}{5pt}
}
{\end{enumerate}}

\begin{document}

\title{Poisson summation formulas involving the sum-of-squares function}

\author{Nir Lev}
\address{Department of Mathematics, Bar-Ilan University, Ramat-Gan 5290002, Israel}
\email{levnir@math.biu.ac.il}

\author{Gilad Reti}
\address{Department of Mathematics, Bar-Ilan University, Ramat-Gan 5290002, Israel}
\email{gilad.reti@gmail.com}

\date{December 12, 2020}
\thanks{Research supported by ISF Grant No.\ 227/17 and ERC Starting Grant No.\ 713927.}

\begin{abstract}
We obtain new Poisson type summation formulas with nodes $\pm \sqrt{n}$ and with weights involving the function $r_k(n)$ that gives the number of representations of a positive integer $n$ as the sum of $k$ squares. Our results extend summation formulas due to Guinand and Meyer that involve the sum-of-three-squares function $r_3(n)$.
\end{abstract}

\maketitle

% =======================================

\section{Introduction} \label{secI1}

\subsection{}

The classical Poisson summation formula asserts that for any
function $f$ on $\R$ satisfying sufficient smoothness and decay assumptions, we have
\begin{equation}
  \label{eqA1.20}
  \sum_{n\in \Z} f(n) = \sum_{n\in \Z} \ft{f}(n)
\end{equation}
where $\ft{f}(t)=\int_{\R} f(x) \exp(-2 \pi i tx) dx$ is the Fourier transform of $f$. The equality \eqref{eqA1.20} holds in particular for any function $f$ from the Schwartz class.
There is an equivalent way to state the Poisson summation formula, by saying that the measure $\mu=\sum_{n\in \Z} \delta_n$
(the sum of unit masses at the integer points) satisfies $\ft{\mu}=\mu$, where $\ft{\mu}$ is the Fourier transform of $\mu$ understood in the sense of temperate distributions.

The problem of which other Poisson type formulas may exist,
was studied by different authors. A recent result proved in \cite{LO13, LO15} states that if $\mu = \sum_{\lam \in \Lam} a_\lam \delta_\lam$ is a pure point measure on $\R$ supported on a uniformly discrete set $\Lam$, and if the Fourier transform $\ft{\mu}$ (in the sense of temperate distributions) is also a pure point measure $\ft{\mu} = \sum_{s \in S} b_s \delta_s$ supported on a uniformly discrete set $S$, then the measure $\mu$ can be obtained from  Poisson's  formula using dilation and a finite number of shifts, modulations, and taking linear combinations. (We recall that a set is said to be \emph{uniformly discrete} if the distance between any two of its points is bounded from below by some positive constant.)

On the other hand, in \cite{LO16} the authors established the existence of pure point measures $\mu$ whose distributional Fourier transform $\ft{\mu}$ is also a pure point measure, and such that the supports of both $\mu$ and $\ft{\mu}$ are \emph{locally finite} sets (i.e.\ they have no finite accumulation points), but the support of $\mu$ contains only finitely many elements of any arithmetic progression. In particular, such a measure $\mu$ cannot be obtained from Poisson's summation formula using the procedures mentioned above (see also \cite{Kol16}).

We note that interest in the subject has been largely inspired by the experimental discovery in the middle of 80's of quasicrystalline materials, i.e.\ nonperiodic  atomic structures that have a discrete diffraction pattern,
see  \cite{Mey95}, \cite{Lag00}. For recent related work see e.g.\ \cite{Mey16}, \cite{Mey17}, \cite{LO17}, \cite{Fav18} and the references therein.

\subsection{}
There is a remarkable summation formula due to Guinand \cite[p.\ 265]{Gui59} which states that if $\varphi$ is an odd Schwartz function on $\R$, and if $\psi = \ft{\varphi}$ is the Fourier transform of $\varphi$, then 
\begin{equation}
  \label{eqA1.3}
  \varphi'(0) + \sum_{n=1}^{\infty} \frac{r_3(n)}{\sqrt{n}} \varphi(\sqrt{n})
  = i \psi'(0) + i \sum_{n=1}^{\infty} \frac{r_3(n)}{\sqrt{n}}  \psi(\sqrt{n})
\end{equation}
where $r_k(n)$ is the number of representations of $n$ as the sum of $k$ squares, that is,
\begin{equation}
  \label{eqA1.2}
  r_k(n) = \# \{m \in \Z^k : |m|^2 = n\}.
\end{equation}

Notice that while \eqref{eqA1.3} holds for odd functions only, one can obtain a summation formula that is valid for an arbitrary Schwartz function $f$ (i.e.\ not necessarily an odd function) by applying \eqref{eqA1.3} to $\varphi(t) := f(t) - f(-t)$ and $\psi(t) := \ft{f}(t) - \ft{f}(-t)$. This yields a nonclassic Poisson type  formula with nodes at the points $\pm \sqrt{n}$ and with weights involving the sum-of-three-squares function $r_3(n)$.

The summation formula \eqref{eqA1.3} can be equivalently stated by saying that  the Fourier transform of the temperate distribution
\begin{equation}
  \label{eqA1.21}
  \sig_3 := -2 \delta'_0 +  \sum_{n=1}^{\infty} \frac{r_3(n)}{\sqrt{n}}  (\delta_{\sqrt{n}} - \delta_{-\sqrt{n}})
\end{equation}
is $-i \sig_3$. In particular, the distribution $\sig_3$ is supported on the locally finite set of points $\pm \sqrt{n}$, and its Fourier transform $\ft{\sig}_3$ is supported on the same set.

Guinand's formula \eqref{eqA1.3} remained little known until it was recently rediscovered by Meyer, see \cite{Mey16}. In that paper, a new proof of Guinand's formula was given, and moreover new examples of Poisson type summation formulas generalizing Guinand's one were obtained.

There is a recent result due to Radchenko and Viazovska \cite{RV19} which states that a Schwartz function $f$ is uniquely determined by the values of $f$ and $\ft{f}$ at the points $\pm \sqrt{n}$ and the values $f'(0)$ and ${\ft{f}}\,'(0)$. Moreover, there exists a linear summation formula that recovers the function $f$ from these values \cite[Theorems 1 and 7]{RV19}. As pointed out in \cite[pp.\ 77--78]{RV19}, one can use this result to derive Guinand's formula \eqref{eqA1.3}.

Remarkably, the Poisson and Guinand formulas (and their linear combinations) are \emph{the only} weighted summation formulas that involve only the values of $f$ and $\ft{f}$ at the points $\pm \sqrt{n}$ and the values $f'(0)$ and ${\ft{f}}\,'(0)$. In other words, if $\sig$ is a temperate distribution such that both $\sig$ and its Fourier transform $\ft{\sig}$ can be expressed as a weighted sum of terms of the form $\delta_0$, $\delta'_0$, $\delta_{\sqrt{n}}$ or  $\delta_{- \sqrt{n}}$ $(n=1,2,3,\dots)$, then $\sigma$ must be a linear combination of the Poisson measure $\mu = \sum_{n \in \Z} \delta_n$ and the Guinand distribution $\sig_3$. This is a consequence of \cite[Theorem 5.2]{Via18}. We note that the Poisson component of $\sig$ corresponds to its even part, while the Guinand component is the odd part of $\sig$.

% =======================================

\section{Results}
\label{sec:results}

\subsection{}
In this paper we obtain new Poisson type summation formulas with nodes $\pm \sqrt{n}$, and with weights involving the function $r_k(n)$ defined by \eqref{eqA1.2} that gives the number of representations of $n$ as the sum of $k$ squares.

For each positive integer $k$, consider a distribution $\sig_k$ on $\R$ defined by
\begin{equation}
  \label{eqA1.1}
  \sig_k := -2 \delta'_0 +  \sum_{n=1}^{\infty} \frac{r_k(n)}{\sqrt{n}}  (\delta_{\sqrt{n}} - \delta_{-\sqrt{n}}).
\end{equation}
We observe that $\sum_{n=1}^{N} r_k(n) n^{-1/2}$ increases polynomially in $N$, which implies that
$\sig_k$ is a temperate distribution. It is obtained from Guinand's distribution \eqref{eqA1.21} by replacing the sum-of-three-squares function $r_3(n)$ with the sum-of-$k$-squares function $r_k(n)$.

It is clear that $\sig_k$ is an odd distribution supported on the set of points $\pm \sqrt{n}$. Guinand's result states that for $k=3$, the Fourier transform $\ft{\sig}_3$ is supported on the same set of points, and in fact, $\ft{\sig}_3 = -i \sig_3$. In the present paper we obtain an extension of this result to all odd values of $k$ greater than $3$. Our result provides an explicit expression for the Fourier transform $\ft{\sig}_k$, which implies in particular that it is again supported on the points $\pm \sqrt{n}$. However we will see that $\ft{\sig}_k$ is not $-i \sig_k$ if $k > 3$.

\subsection{}
For simplicity we first state the result for $k=5$.

\begin{thm}
  \label{thmA1}
  Let $\varphi$ be an odd Schwartz function on $\R$, and let $\psi = \ft{\varphi}$  be the Fourier transform of $\varphi$. Then
  \begin{equation}
    \label{eqA1.8}
    \varphi'(0) +  \sum_{n=1}^{\infty} \frac{r_5(n)}{\sqrt{n}} \varphi(\sqrt{n})
    =  - \frac{i}{6 \pi} \psi'''(0) + \frac{i}{2 \pi }
    \sum_{n=1}^{\infty} \frac{r_5(n)}{n^{3/2}} \Big[
      \psi(\sqrt{n}) - \sqrt{n} \, \psi'(\sqrt{n}) \Big].
  \end{equation}
\end{thm}

The result establishes a new summation formula with nodes $\pm \sqrt{n}$ and with weights involving the sum-of-five-squares function $r_5(n)$. We observe that the right hand side of \eqref{eqA1.8} involves not only the values of $\psi$, but also those of its derivative $\psi'$, at the points $\sqrt{n}$ $(n=1,2,3,\dots)$. This is not the case in Guinand's formula \eqref{eqA1.3}.

\thmref{thmA1} can be equivalently stated by saying that the Fourier transform of the  distribution  $\sig_5$ is given by
\begin{equation}
  \label{eqA1.6}
  \ft{\sig}_5 = - \frac{i}{3 \pi }  \delta'''_0 - \frac{i}{2 \pi}
  \sum_{n=1}^{\infty} \frac{r_5(n)}{n^{3/2}} \Big[
   (\delta_{\sqrt{n}} - \delta_{-\sqrt{n}})  +
   \sqrt{n} \, (\delta'_{\sqrt{n}} + \delta'_{-\sqrt{n}}) \Big].
\end{equation}
As a consequence we obtain that $\ft{\sig}_5$ is also supported on the set of points $\pm \sqrt{n}$.

\subsection{}
\thmref{thmA1}  is a special case of the next result, which yields a different summation formula for each $k = 3,5,7,9,\dots$.
To state the result, we define the numbers
  \begin{equation}
    \label{eqA1.11}
    \alpha_k :=  \frac{(-1)^{(k-3)/2}}{(k-2)!!} \Big(\frac1{2\pi}\Big)^{(k-3)/2}, \quad
    \beta_{j,k} :=  \frac{(-1)^j \, (k-j-3)!}{j! (k-2j-3)!!} \Big(\frac1{2\pi}\Big)^{(k-3)/2}
  \end{equation}
where $k \geq 3$ is an odd integer, and $0 \leq j \leq (k-3)/2$.

\begin{thm}
  \label{thmA3}
  Let $\varphi$ be an odd Schwartz function on $\R$, and let $\psi = \ft{\varphi}$ be the Fourier transform of $\varphi$. For each odd integer $k \geq 3$ we have
\[
    \varphi'(0) +  \sum_{n=1}^{\infty} \frac{r_k(n)}{\sqrt{n}} \varphi(\sqrt{n})
    =  i \alpha_k \, \psi^{(k-2)}(0) +
   i \sum_{n=1}^{\infty}
    \frac{r_k(n)}{n^{(k-2)/2}}
    \Big[  \sum_{j=0}^{(k-3)/2}
      \beta_{j,k} \, n^{j/2} \,  \psi^{(j)}(\sqrt{n}) \Big],
\]
  where the coefficients $\alpha_k$ and $\beta_{j,k}$ are given by \eqref{eqA1.11}.
\end{thm}

We thus obtain, for each $k=3,5,7,9,\dots$, a new Poisson type summation formula with weights involving the sum-of-$k$-squares function $r_k(n)$. The nodes remain at the points $\pm \sqrt{n}$, but the number of derivatives at each node depends on $k$ and it increases with $k$.  We observe that \eqref{eqA1.3} and \eqref{eqA1.8} are the special cases obtained for $k=3$ and $5$.

The result is equivalent to saying that for each odd integer $k \geq 3$, the Fourier transform of the  distribution  $\sig_k$ defined by \eqref{eqA1.1}  is given by
\begin{equation}
  \label{eqA1.7}
  \ft{\sig}_k = 2i \alpha_k   \delta^{(k-2)}_0
  -  i \sum_{n=1}^{\infty} \frac{r_k(n)}{n^{(k-2)/2}}
  \Big[ \sum_{j=0}^{(k-3)/2}
   \beta_{j,k} \, n^{j/2} \,
  \Big( (-1)^j \, \delta^{(j)}_{\sqrt{n}} - \delta^{(j)}_{-\sqrt{n}}\Big)
  \Big].
\end{equation}
In particular, $\sig_k$ is a temperate distribution supported on the locally finite set of points $\pm \sqrt{n}$, and its Fourier transform $\ft{\sig}_k$ is supported on the same set.

\subsection{}
Our proof of the results above is inspired by Meyer's proof in \cite{Mey16}
of Guinand's formula \eqref{eqA1.3}. We will obtain \thmref{thmA3}
as a consequence of the following result, 
which generalizes \cite[Theorem 7]{Mey16}.

\begin{thm}
  \label{thmA5}
  Let $\mu = \sum_{\lam \in \Lam} a(\lam) \delta_\lam$
  be a measure on $\R^k$, $k \in \{3,5,7,9,\dots\}$, whose
  support $\Lam$ is a locally finite set. Suppose that
  $\mu$ is a temperate distribution, and that its Fourier
  transform $\ft{\mu} = \sum_{s \in S} b(s) \delta_s$
  is also a measure, supported on a locally finite set $S$.
  Assume that $0 \notin \Lam$, $0 \notin S$, and let
  $\sig$ be a measure on $\R$  defined by
  \begin{equation}
    \label{eqA5.1}
    \sig := \sum_{\lam \in \Lam} \frac{a(\lam)}{|\lam|}
    ( \delta_{|\lam|} -  \delta_{-|\lam|}).
  \end{equation}
  Then $\sig$ is a temperate distribution whose one-dimensional Fourier transform is
  \begin{equation}
    \label{eqA5.2}
    \ft{\sig} =
    -i \sum_{s \in S} \frac{b(s)}{|s|^{k-2}}
    \Big[ \sum_{j=0}^{(k-3)/2}
      \beta_{j,k} \, |s|^j \,
      \Big( (-1)^j \, \delta^{(j)}_{|s|} - \delta^{(j)}_{-|s|}\Big)
      \Big],
  \end{equation}
  where the coefficients $\beta_{j,k}$ are defined in \eqref{eqA1.11}.
\end{thm}

The result in \cite[Theorem 7]{Mey16} corresponds to the case  $k=3$ 
in \thmref{thmA5}.

We observe that the measure $\sig$ in \eqref{eqA5.1} is supported on the 
locally finite subset $\{\pm |\lam| : \lam \in \Lam\}$ of the real line.
 It follows from \thmref{thmA5} that the Fourier
transform $\ft{\sig}$ is a distribution supported on the (also locally finite) set
$\{\pm |s| : s \in S\}$.

In \secref{sec:newpoisson} we will give a more general version of \thmref{thmA5}, in which the origin is allowed to belong to the supports $\Lam, S$ of the measures $\mu$ and $\ft{\mu}$ (\thmref{thmA8}).

\subsection{}
As an example, \thmref{thmA5} applies to the measure
$\mu_{(\eta,\xi)} = \sum_{m \in \Z^k} e^{2 \pi i \dotprod{m}{\xi}} \delta_{m + \eta}$ where $\eta$ and $\xi$ are two vectors in $\R^k \setminus \Z^k$.
It follows from Poisson's formula that the Fourier transform of $\mu_{(\eta,\xi)}$ is given by $\ft{\mu}_{(\eta,\xi)} = e^{- 2 \pi i \dotprod{\eta}{\xi}} \mu_{(\xi,-\eta)}$. Hence applying \thmref{thmA5}  to this measure yields
the following generalization of \cite[Theorem 5]{Mey16}.

\begin{corollary}
  \label{corA6}
  Let $\eta$ and $\xi$ be two vectors in $\R^k \setminus \Z^k$,
  $k \in \{3,5,7,9,\dots\}$. Then the   one-dimensional Fourier transform of the measure
  \begin{equation}
    \label{eqA6.1}
    \sig := \sum_{m \in \Z^k} \frac{e^{2 \pi i \dotprod{m}{\xi}}}{|m+\eta|}
    ( \delta_{|m + \eta|} -  \delta_{- |m + \eta|} )
  \end{equation}
  is the distribution
  \begin{equation}
    \label{eqA6.2}
    \ft{\sig} = -i e^{- 2 \pi i \dotprod{\eta}{\xi}}
    \sum_{m \in \Z^k} \frac{e^{- 2 \pi i \dotprod{m}{\eta}}}{|m + \xi|^{k-2}}
    \Big[ \sum_{j=0}^{(k-3)/2}
      \beta_{j,k} \, |m + \xi|^j \,
      \Big( (-1)^j \, \delta^{(j)}_{|m + \xi|} - \delta^{(j)}_{-|m + \xi|}\Big)
      \Big].
  \end{equation}
\end{corollary}

\subsection{}
In order to prove \thmref{thmA3}, we will give in 
\secref{sec:newpoisson} a more general version of \thmref{thmA5}
(\thmref{thmA8}), in which the supports $\Lam, S$ of the measures 
$\mu$ and $\ft{\mu}$ are allowed to contain the origin.
\thmref{thmA3} will then
be deduced by applying this result to the Poisson measure
$\mu = \sum_{m \in \Z^k} \delta_m$ which satisfies
$\ft{\mu} = \mu$.

We observe that (in the same spirit
as Meyer's proof of Guinand's formula) there exist 
distributional limits for \eqref{eqA6.1} and \eqref{eqA6.2} as
$\eta$ and $\xi$ tend to zero, and these
limits coincide with \eqref{eqA1.1} and \eqref{eqA1.7} respectively.

The rest of the paper is devoted to the proofs of the results stated above.

% =======================================

\section{Preliminaries}

In this section we briefly recall some preliminary background
in the theory of Schwartz distributions
(see \cite{Rud91} for more details)
 and obtain a few basic results that will be
used later on.

\subsection{}
We denote by $\mathcal{D}(\R^k)$ the space of all infinitely smooth, compactly 
supported functions on $\R^k$. A sequence of functions $\varphi_j$ in 
$\mathcal{D}(\R^k)$ is said to converge in the space   $\mathcal{D}(\R^k)$ if (i) there
exists a compact set $K \sbt \R^k$ such that $\supp(\varphi_j) \subset K$ for all $j$;
and (ii) for each multi-index $m = (m_1,\dots,m_k)$,
the sequence of partial derivatives $\partial^m \varphi_j$ converges uniformly on $K$
 as $j \to \infty$. A \emph{distribution} is a linear functional on $\mathcal{D}(\R^k)$
which is continuous with respect to the convergence in the space $\mathcal{D}(\R^k)$.

We denote by $\dotprod{\sig}{\varphi}$ the action of a 
distribution  $\sig$  on  a function $\varphi$ belonging to $\mathcal{D}(\R^k)$.

The \emph{Schwartz space} $\mathcal{S}(\R^k)$ consists of all infinitely smooth
functions $\varphi$ on $\R^k$ such that for each $n \geq 0$ and 
each multi-index $m = (m_1,\dots,m_k)$, the norm
\[
\|\varphi\|_{n,m} := \sup_{x\in \R^k} |x|^n |\partial^m \varphi(x)|
\]
is finite. A sequence $\varphi_j$ is said to converge in the space 
$\mathcal{S}(\R^k)$ if it converges with respect to each one
of the norms $\| \cdot \|_{n,m}$.
A \emph{temperate distribution} is a distribution that can be extended
to a linear functional on the space $\mathcal{S}(\R^k)$ in such a 
way that the extension is continuous with respect to the convergence in the space 
$\mathcal{S}(\R^k)$.

If $\sig$   is a temperate distribution then the
notation $\dotprod{\sig}{\varphi}$ can be extended
to denote the action of $\sig$ on any   
function $\varphi$ from the Schwartz space $\mathcal{S}(\R^k)$.

If $\varphi$ is a Schwartz function on $\R^k$ then its Fourier transform
is defined by
\[
\ft \varphi (\xi)=\int_{\R^k} \varphi (x) \, e^{-2\pi i\langle \xi,x\rangle} dx.
\]
The Fourier transform $\ft\sig$ of a temperate distribution $\sig$ is defined by 
$\dotprod{\ft\sig}{\varphi} = \dotprod{\sig}{\ft\varphi}$.

A distribution  $\sig$  on $\R$  is said to be \emph{even}
(respectively \emph{odd}) if we have
$\dotprod{\sig}{\varphi} = 0$ for every odd 
(respectively even)  function $\varphi \in \mathcal{D}(\R)$.

\subsection{}
The following result is a version of the well-known Hadamard Lemma,
which states that if $f$ is a function of the class $C^r(\R)$
 (that is, $f$ is  $r$-times continuously differentiable)
such that $f(0)=0$, then $f(t)=t g(t)$
where $g$ is a function
of the class $C^{r-1}(\R)$.

\begin{lem}
  \label{lemC2}
  Let $f$ be a function in the Schwartz  space $\mathcal{S}(\R)$,
$f(0)=0$.  Then 
 \begin{equation}
\label{eq:hadam}
f(t) = t g(t), \quad t\in \R,
 \end{equation}
where $g$ is  a function in $\mathcal{S}(\R)$.
Moreover, the mapping which takes $f$ to $g$ is a continuous linear mapping 
from the space of Schwartz  functions vanishing
at the origin into $\mathcal{S}(\R)$.
\end{lem}

\begin{proof}
Define
$g(t) := \int_{0}^1 f'(tu)  du$, 
 then $g$ is an infinitely smooth function
and satisfies condition \eqref{eq:hadam}. The fact that $g(t) = f(t)/t$ $(t \neq 0)$ implies that $\| g \|_{n,m}$ is finite for every $n$ and $m$, and thus $g \in \mathcal{S}(\R)$.
The mapping which takes $f$ to $g$ is thus a linear mapping
from the space of Schwartz  functions vanishing
at the origin into $\mathcal{S}(\R)$.
Now suppose that $f_j \to f$ in $\mathcal{S}(\R)$. If we denote by $g_j$ and $g$
the functions corresponding to $f_j$ and $f$ respectively by the mapping, then $g_j \to g$ pointwise. The continuity of the mapping thus follows from the closed graph theorem (see \cite[Theorem 2.15]{Rud91}).
\end{proof}

\subsection{}

\begin{lem}
  \label{lemC3}
  Let $f$ be an even function in the Schwartz space $\mathcal{S}(\mathbb{R})$,
and let $F$ be a radial function on $\R^k$ defined by $F(x) := f(|x|)$,
$x\in \R^k$. Then $F$ belongs to the Schwartz space $\mathcal{S}(\R^k)$.
Moreover, the mapping which takes $f$ to $F$ is a continuous linear mapping 
from the space of even functions in $\mathcal{S}(\R)$ into $\mathcal{S}(\R^k)$.
\end{lem}

\begin{proof}
Since $f$ is an even function we have $f'(0)=0$. Hence by \lemref{lemC2}
there is a function $g \in \mathcal{S}(\R)$ such that
$f'(t) = t g(t)$, $t \in \R$. We claim that 
  \begin{equation}
\label{eq:lemC3.1}
   \frac{\partial F}{\partial x_i}(x) = g(|x|) x_i, \quad x= (x_1, \dots, x_k)
  \end{equation}
for each $1\le i\le k$.
Indeed, this follows from the chain rule for $x \neq 0$,
while for $x=0$ we simply observe that both sides of \eqref{eq:lemC3.1} vanish.
Since the right-hand side of \eqref{eq:lemC3.1} is a continuous function
for each $i$, we obtain that $F$ is a function in the class $C^1(\R^k)$.

Applying the same considerations to the (even) function $g$ in place of $f$, 
we obtain that the function $G(x) := g(|x|)$ is also in the class $C^1(\R^k)$.
This implies that the right-hand side of \eqref{eq:lemC3.1} belongs
to $C^1(\R^k)$ for each $i$, and it follows that $F \in C^2(\R^k)$.
Continuing in this way we obtain that $F \in C^r(\R^k)$ for every
$r$, hence $F$ is infinitely smooth.

It is straightforward to check that
$\| F \|_{n,m}$ is finite for every $n$ and every multi-index $m = (m_1,\dots,m_k)$,
 which implies that $F \in \mathcal{S}(\R^k)$.
The mapping which takes $f$ to $F$ is thus a linear mapping
from the space of even functions 
in $\mathcal{S}(\R)$ into $\mathcal{S}(\R^k)$.

Finally, suppose that $f_j \to f$ in $\mathcal{S}(\R)$. If $F_j$ and
$F$ are the functions  corresponding to $f_j$ and $f$ respectively by the mapping, then 
$F_j \to F$ pointwise. This establishes the continuity of the mapping again
as a consequence of the closed graph theorem.
\end{proof}

% =======================================

\section{Fourier transform of a radial function in odd dimensions}
\label{sec:radial}

\subsection{}
Let $f$ be an even Schwartz function on $\R$, and suppose that we use $f$ to construct a radial function $F_k$ on $\R^k$, defined by $F_k(x) := f(|x|)$. The $k$-dimensional Fourier transform $\ft{F}_k$ of the function $F_k$ is again a radial function, so there is an even function $f_k$ on $\R$ such that $\ft{F}_k(\xi) = f_k(|\xi|)$. It turns out that if the dimension $k$ is an odd integer, then there exists an explicit expression for the new function $f_k$ in terms of the one-dimensional Fourier transform $\ft{f}$ and its derivatives. This expression is  given in the following result:

\begin{thm}
  \label{thmD1}
  Let $f$ be an even Schwartz function on $\R$, and let $ k \geq 3 $ be an odd integer. Define a radial function $F_k$ on $\R^k$ by $F_k(x) := f(|x|)$. Then the $k$-dimensional Fourier transform of $F_k$ is given by
  \begin{equation}
    \label{eqD2}
    \ft{F}_k(\xi) = -\frac{1}{2\pi|\xi|^{k-1}}
    \sum_{j=0}^{(k-3)/2}
    \beta_{j,k} \, |\xi|^{j+1} \, \ft{f}^{\:(j+1)}(|\xi|)  , 
	\quad \xi \in \R^k\setminus\{0\}
  \end{equation}
  where the coefficients $\beta_{j,k}$ are defined as in \eqref{eqA1.11}, and where $\ft{f}$ is the one-dimensional Fourier transform of the function $f$.
\end{thm}

This result was obtained in \cite[Corollary 1.2]{GT13}.  In this paper the authors also give an expression for \(\ft{F}_k\) for even dimensions $k$,  but in this case the expression is not ``explicit'', i.e.\ it is not given in terms of the one-dimensional Fourier transform $\ft{f}$. Below we give a proof of \thmref{thmD1} that is different from the one in \cite{GT13} (our method can also be used to establish the corresponding result for even dimensions $k$).

The assumption in \thmref{thmD1} that \(f\) is a Schwartz function is not essential for the conclusion. The result remains valid (with the same proof) e.g.\ for any  \(f\) such that \((1+|t|)^{k-1} f(t)\) is in \(L^1(\R)\), a condition which ensures that \(F_k\) is a function in \(L^1(\R^k)\).

\subsection{}
We now turn to the proof of \thmref{thmD1}. We begin with a simple claim about a recurrence relation satisfied by the Fourier transforms of the surface measures on the unit spheres in \(\R^k\).
\begin{lem}
  \label{lemD2}
  Let $\mu_k$ denote the surface measure on the $(k-1)$-dimensional unit sphere in $\R^k$. The Fourier transform $\ft{\mu}_k$ is a radial function, hence there is an even function $s_k$ satisfying $\ft{\mu}_k(\xi) = s_k(|\xi|)$, $\xi \in \R^k$. Then the
recurrence relation 
\begin{equation}
  \label{eqD2_1}
s_k(t) = (2 \pi t^2)^{-1} ((k-4)s_{k-2}(t) - 2\pi s_{k-4}(t)), \quad t \neq 0,
\end{equation}
is valid for each (even or odd) integer $k \geq 5$.
\end{lem}

This follows from the recurrence relation satisfied by the Bessel functions:

\begin{proof}[Proof of \lemref{lemD2}]
  It is well-known (see for instance \cite[p.\ 154]{SW71}) that the Fourier transform of the surface measure $\mu_k$ satisfies $\hat{\mu}_k(\xi) = 2 \pi |\xi|^{-(k-2)/2} J_{(k-2)/2}(2 \pi |\xi|)$, where $J_{\nu}$ is the Bessel function of order $\nu$. It is also known that the Bessel functions satisfy the recurrence relation $2\nu z^{-1} J_{\nu}(z) = J_{\nu-1}(z) + J_{\nu+1}(z)$, see \cite[Section 5.3]{Leb72}. For $\nu = (k-4)/2$ this relation yields \eqref{eqD2_1}, as one can verify in a straightforward manner.
\end{proof}

\subsection{}
Equipped with  \lemref{lemD2}, we can now give the proof of \thmref{thmD1}. 

\begin{proof}[Proof of \thmref{thmD1}]
Recall that given an even Schwartz function $f$ on $\R$ we define a radial function $F_k$ on $\R^k$ by $F_k(x) := f(|x|)$. We observe that \(F_k\) is in the Schwartz space, by \lemref{lemC3}.
Since the Fourier transform of a radial function is also radial, there is an even function $f_k$ such that $\ft{F}_k(\xi) = f_k(|\xi|)$, $\xi \in \R^k$. 
 This allows us to define a linear operator \(A_k\) on the space of even  Schwartz functions on $\R$, given by \(A_kf := f_k\).

Using $k$-dimensional spherical integration we have
\begin{equation}
  \label{eqD3_1}
  \ft{F}_k(\xi) = \int_{\mathbb{R}^k} f(|x|) \, e^{-2 \pi i \langle \xi , x \rangle} \, dx = \int_0^\infty f(r) \ft{\mu}_k(r\xi) r^{k-1} \, dr,
\end{equation}
where $\mu_k$ is the surface measure on the $(k-1)$-dimensional unit sphere in $\R^k$. If we let \(s_k\) be the function defined in \lemref{lemD2}, then this yields the formula
\begin{equation}
  \label{eqD1_5}
A_kf(t) = \int_0^\infty f(r) s_k(rt) r^{k-1} \, dr.
\end{equation}
We note that this formula is well-known and can be found e.g.\ in \cite[p.\,155]{SW71}.

Next, we define another family of operators $B_k$ $(k=1,3,5,7,\dots)$ on the space of even  Schwartz functions on $\R$. We let  $B_1f(t):=\hat{f}(t)$; while for $k=3,5,7,\dots$ we define 
  \begin{equation}
    \label{eqD2_6}
    B_k f(t) := -\frac{1}{2\pi t^{k-1}}
    \sum_{j=0}^{(k-3)/2}
    \beta_{j,k} \, t^{j+1} \, \ft{f}^{\:(j+1)}(t), \quad t \neq 0
  \end{equation}
(we do not specify the value of $B_k f$ at the point $t=0$).
The assertion in \thmref{thmD1} may thus be reformulated by saying that $A_k f(t) = B_k f(t)$, $t \neq 0$, for every even function $f$ in $\mathcal{S}(\R)$.
A straightforward calculation whose details are omitted shows that the operators $\{B_k\}$, $k=1,3,5,7,\dots$, satisfy the recurrence relation
\begin{equation}
  \label{eqD3}
  B_kf(t) = \frac{k-4}{2 \pi t^2}  \, B_{k-2}f(t) - \frac{1}{t^2} B_{k-4}(t^2 f(t)).
\end{equation}
In order to prove our claim, we will show that the sequence $\{A_k\}$, $k=1,3,5,7,\dots$ satisfies the same recurrence relation as \eqref{eqD3} with the same initial conditions, and the conclusion will thus follow by induction on $k$.

Indeed, using \eqref{eqD1_5} together with the recurrence relation \eqref{eqD2_1} we get
\[
 A_kf (t)  = \frac{k-4}{2 \pi t^2} \int_0^\infty f(r)r^{k-3}s_{k-2}(rt) \, dr - \frac{1}{t^2} \int_0^\infty (r^2 f(r))r^{k-5} s_{k-4}(rt) \, dr,
\]
and thus we see that
\begin{equation}
  \label{eqD4}
  A_kf(t)=\frac{k-4}{2 \pi t^2}A_{k-2}f(t) - \frac{1}{t^2}A_{k-4}(t^2f(t)).
\end{equation}

It remains to check the base cases $k=1$ and $3$.
If $k = 1$ then $F_1(x)= f(x)$ and so $A_1f(t)=\ft{f}(t)=B_1 f(t)$.
For $k = 3$ we have $s_3(t)=2\sin(2 \pi t)/t$, and \eqref{eqD1_5} implies that
\[
  A_3f(t) =  \frac{2}{t} \int_0^\infty rf(r) \sin(2 \pi r t) \, dr.
\]
But since $f$ is an even function, this yields
\[
  A_3f(t) = \frac{1}{t} \int_{\R} rf(r) \sin(2 \pi r t) \, dr = -\frac{1}{2 \pi t} \ft{f}\;'(t) = B_3f(t),
\]
as required. This concludes the proof of \thmref{thmD1}. 
\end{proof}

\begin{remark}
  \label{remD1}
  \thmref{thmD1} provides an expression for \(\ft{F}_k(\xi)\) for all \(\xi \in \R^k\setminus\{0\}\), that is, except for the value at $\xi=0$. This value may be not easy to find directly from \eqref{eqD2} by continuity. Fortunately, using a simple argument one can obtain also an explicit expression for \(\ft{F}_k(0)\). Indeed, since $f$ is an even function, it follows from \eqref{eqD3_1} that
  \[
    \ft{F}_k(0) = \ft{\mu}_k(0) \int_0^\infty f(r) r^{k-1} \, dr
    = \tfrac1{2}  \ft{\mu}_k(0) \int_{\R} f(r) r^{k-1} \, dr.
  \]
We observe that \(\ft{\mu}_k(0)\) is the total surface area of the unit sphere in \(\R^k\), which (since $k$ is odd) is known to be equal to \(2(2\pi)^{(k-1)/2}/(k-2)!!\). This yields the expression
\begin{equation}
  \label{eqD1.5}
    \ft{F}_k(0) = -\frac{\alpha_k}{2\pi}\, \ft{f}^{\:(k-1)}(0),
\end{equation}
  where \(\alpha_k\) is defined as in \eqref{eqA1.11}. 
\end{remark}

\subsection{}
As an application of \thmref{thmD1} we consider the
Fourier transform of the surface measure 
on the $(k-1)$-dimensional unit sphere in $\R^k$.
It is well-known that for odd dimensions $k$, the
Fourier transform of this measure is expressible
as a finite sum using powers, sines and cosines
(this can be inferred e.g.\ from the recurrence relation 
given in \lemref{lemD2}).  \thmref{thmD1} allows us
to derive a closed-form expression  for the
Fourier transform of the surface measure
(such formulas are in fact known, see e.g.\ \cite{KF49}).

\begin{thm}
  \label{thmD2}
  Let $\mu_k$ be the surface measure on the $(k-1)$-dimensional unit sphere in $\R^k$. If $k$ is an odd integer, $k \geq 3$, then
 the Fourier transform of the measure $\mu_k$ is
\begin{equation}
  \label{eqD2.1}
    \ft{\mu}_k (\xi) =  \frac{2}{|\xi|^{k-2}}
    \sum_{j=0}^{(k-3)/2}
    \beta_{j,k} \, (2 \pi |\xi|)^{j} \sin \big(2 \pi |\xi| + \tfrac{\pi}{2} j \big), 
	\quad \xi \in \R^k\setminus\{0\},
\end{equation}
 where the coefficients $\beta_{j,k}$ are as in \eqref{eqA1.11}.
\end{thm}

\begin{proof}
Choose a non-negative,  infinitely smooth, compactly 
supported, even function $\varphi$ on $\R$,  such that 
 $\int_{\R} \varphi(t) dt = 1$. For each $\eps>0$ define
the function
\[
f_\eps(t) := \eps^{-1} \big( \varphi((t-1)/\eps) + \varphi((t+1)/\eps) \big),
\]
then $f_\eps$ is an even function and $f_\eps \to \delta_{1}  + \delta_{-1}$  as
$\eps \to 0$ in the sense of distributions.
In particular it follows from \eqref{eqD1_5} that 
$A_k f_\eps (t)$ tends to $s_k(t)$ pointwise as $\eps \to 0$,
where  $s_k$ is the even function on $\R$ which satisfies
 $\ft{\mu}_k(\xi) = s_k(|\xi|)$, $\xi \in \R^k$. 

Recall that in the proof of \thmref{thmD1} we have shown that 
$A_k f(t) = B_k f(t)$, $t \neq 0$, for every even function $f$ in the Schwartz space,
where $B_k f(t)$ is defined as in \eqref{eqD2_6}. We can therefore
obtain the value of $s_k(t)$ for $t \neq 0$ as the limit of
$B_k f_\eps(t)$ as $\eps  \to 0$.

First we observe that
$\ft{f}_\eps(t) = 2 \, \ft{\varphi}(\eps t) \cos (2 \pi t)$.
Then by the Leibniz rule we have
\[
     \ft{f}_\eps^{\:(j+1)}(t) = 
2 \, \ft{\varphi}(\eps t) \big\{ \cos (2 \pi t) \big\}^{(j+1)} + \cdots,
  \]
 where the omitted terms tend to zero as $\eps \to 0$. This implies that
  \begin{equation}
    \label{eqD2_20}
\lim_{\eps \to 0}
     \ft{f}_\eps^{\:(j+1)}(t) = 
2  \big\{ \cos (2 \pi t) \big\}^{(j+1)} =
-2 (2 \pi)^{j+1} \sin \big(2 \pi t  + \tfrac{\pi}{2} j  \big).
  \end{equation}
Finally, combining \eqref{eqD2_6} and \eqref{eqD2_20} yields
  \begin{equation}
s_k(t) = \lim_{\eps \to 0}
    B_k f_\eps(t) = \frac{2}{t^{k-2}}
    \sum_{j=0}^{(k-3)/2}
    \beta_{j,k} \, (2 \pi t)^{j} \, \sin \big(2 \pi t  + \tfrac{\pi}{2} j  \big),
 \quad t \neq 0,
  \end{equation}
which is equivalent to the assertion in \eqref{eqD2.1}.
\end{proof}

\begin{remark}
  \label{remD2_3}
We could have proved \thmref{thmD1} and
\thmref{thmD2} also  the other way around, that is, 
at the  first stage we could have proved
\thmref{thmD2}  by induction on $k$,
showing that \eqref{eqD2.1} is valid
using the recurrence relation \eqref{eqD2_1}
and checking the base cases $k=3$ and $5$;
 while at the second stage we could have
derived \thmref{thmD1} 
based on \eqref{eqD3_1} and \eqref{eqD2.1}.
We have opted for giving a direct proof of
 \thmref{thmD1}, since it is the result that will
be used in the next section.
\end{remark}

\begin{remark}
  \label{remD2_5}
The formula \eqref{eqD2.1} for  the Fourier transform of the surface 
measure $\mu_k$ is closely related to the \emph{Bessel polynomials}
$\theta_n(z)$. This is a sequence of polynomials satisfying 
the recurrence relation 
$\theta_n(z) = (2n-1) \theta_{n-1}(z) + z^2 \theta_{n-2}(z)$
with the initial conditions
$\theta_0(z)=1$ and $\theta_1(z)=z+1$.
The $n$'th polynomial in the sequence can be written
explicitly as $\theta_n(z) = (2\pi)^n \sum_{j=0}^{n} 
\beta_{j,k} (-z)^j$, where $k = 2n+3$ and where
the coefficients  $\beta_{j,k}$ are the same as in \eqref{eqA1.11}.
We can therefore rewrite \eqref{eqD2.1} in the form
\begin{equation}
  \label{eqD2_5_1}
    \ft{\mu}_k (\xi) =  \frac{2}{|\xi|^{k-2}} \, \Im \left\{
\frac{\theta_n(-2\pi i |\xi|)}{(2 \pi)^n} \, \exp(2\pi i |\xi|)  \right\},
\quad n = \frac{k-3}{2},
\end{equation}
which clarifies the relation between the Bessel polynomials and
the Fourier transform  of the surface measure $\mu_k$ 
for odd dimensions $k$. We refer to \cite{Gro78}
for a comprehensive survey on the Bessel polynomials
(see also \cite{KF49}).
\end{remark}

% =======================================

\section{New Poisson type summation formulas}
\label{sec:newpoisson}

In this section we prove the results stated in \secref{sec:results}.
We also state and prove \thmref{thmA8}  that generalizes
\thmref{thmA5} to the case where the origin is allowed to belong 
to the supports $\Lam, S$ of the measures $\mu$ and $\ft{\mu}$.

\subsection{}
First we will prove \thmref{thmA5} which generalizes
\cite[Theorem 7]{Mey16} to all odd values of $k$ greater than
$3$.  Meyer's proof in the case  $k=3$ was based on the
following observation: Let a radial function $G$ on $\R^3$ 
be defined by $G(x) :=  f(|x|) / |x|$, where $f$ is an odd Schwartz 
function on $\R$. Then the three-dimensional 
Fourier transform of the function $G$ is given by
$\ft{G}(\xi) = i \ft{f}(|\xi|) / |\xi|$.

The proof of \thmref{thmA5} for an arbitrary odd integer $k \geq 3$ will follow a similar 
line, where in this case we will use the connection given in \thmref{thmD1} 
between the one-dimensional Fourier transform of an even Schwartz function on \(\R\) and the $k$-dimensional Fourier transform of a radial function on $\R^k$.

\begin{proof}[Proof of \thmref{thmA5}]
First we show that the distribution $\sig$ defined by
\eqref{eqA5.1} is indeed a temperate distribution. If
$f$ is a function in the Schwartz  space $\mathcal{S}(\R)$, 
then by \lemref{eq:hadam} we have $f(t) - f(-t) = t g(t)$ 
where $g$ is also in $\mathcal{S}(\R)$. Then $g$ is
an even function and so by \lemref{lemC3} the function
$G(x) := g(|x|)$, $x \in \R^k$, is in the Schwartz  
space $\mathcal{S}(\R^k)$.
Moreover, the mapping which takes $f$ to $G$ is a 
continuous linear mapping from the space $\mathcal{S}(\R)$
into $\mathcal{S}(\R^k)$.
Observe that if $f$ is compactly supported then we have
\begin{equation}
  \label{eqE1.1}
\dotprod{\sig}{f} =  \sum_{\lam \in \Lam} \frac{a(\lam)}{|\lam|}
    ( f(|\lam|) -  f(-|\lam|)) = 
  \sum_{\lam \in \Lam} a(\lam) G(\lam)=
\dotprod{\mu}{G}.
\end{equation}
Since the measure $\mu$ is assumed to be a temperate distribution,
the mapping $f \mapsto \dotprod{\mu}{G}$ thus
defines a continuous  linear functional on the Schwartz space 
$\mathcal{S}(\R)$ which agrees with $\sig$ on the space 
$\mathcal{D}(\R)$ of infinitely smooth, compactly 
supported functions. This means that 
$\sig$ is a temperate distribution.

Next we show that the Fourier transform of $\sig$ 
 is given by \eqref{eqA5.2}. Since $\sig$ is an odd
distribution, the Fourier transform of $\ft{\sig}$ is $-\sig$.
Hence \eqref{eqA5.2} is equivalent to saying that
  \begin{equation}
  \label{eqE1.6}
    \dotprod{\sig}{f} =
    i \sum_{s \in S} \frac{b(s)}{|s|^{k-2}}
    \Big[ \sum_{j=0}^{(k-3)/2}
       \beta_{j,k} \, |s|^j \,
      \Big( \ft{f}^{\:(j)}(|s|) -  (-1)^j \ft{f}^{\:(j)}(-|s|) \Big)
      \Big]
  \end{equation}
for every Schwartz function $f$ such that $\ft{f}$ is
compactly supported. If we keep using
$g$ and $G$ to denote the two functions related to $f$ as above, 
then by \thmref{thmD1} we have
\begin{equation}
  \label{eqE1.3}
    \ft{G}(\xi) = -\frac{1}{2\pi|\xi|^{k-1}}
    \sum_{j=0}^{(k-3)/2}
    \beta_{j,k} \, |\xi|^{j+1} \, \ft{g}^{\:(j+1)}(|\xi|)  , 
	\quad \xi \neq 0,
\end{equation}
while the property $t g(t) = f(t) - f(-t)$ implies that
\begin{equation}
  \label{eqE1.4}
\ft{g}^{\:(j+1)}(t)  = -2\pi i \big(\ft{f}^{\:(j)}(t) - (-1)^j \ft{f}^{\:(j)}(-t)\big).
\end{equation}
It now follows using \eqref{eqE1.3} and \eqref{eqE1.4} that
the right-hand side of \eqref{eqE1.6} is equal to
\begin{equation}
  \label{eqE1.7}
 \sum_{s \in S} b(s) \ft{G}(s) =  \dotprod{\ft{\mu}}{\ft{G}}
= \dotprod{\mu}{G} = \dotprod{\sig}{f}
\end{equation}
as we had to show. (We note that in the second equality in \eqref{eqE1.7} we used
the fact that the Fourier transform of $\ft{G}$ is $G$,
which follows from the property $G(-x)=G(x)$.)
\end{proof}

\subsection{}
\corref{corA6} now follows easily as a special case of \thmref{thmA5}. 

 \begin{proof}[Proof of \corref{corA6}]
We apply \thmref{thmA5} to the measure
\[
\mu_{(\eta,\xi)} = \sum_{m \in \Z^k} e^{2 \pi i \dotprod{m}{\xi}} \delta_{m + \eta}
\]
where the vectors $\eta$ and $\xi$ are in $\R^k \setminus \Z^k$. 
By Poisson's summation formula, the Fourier transform of 
$\mu_{(\eta,\xi)}$ is the measure
 $\ft{\mu}_{(\eta,\xi)} = e^{- 2 \pi i \dotprod{\eta}{\xi}} \mu_{(\xi,-\eta)}$.
Since neither $\eta$ nor $\xi$ belongs to $\Z^k$, the origin lies in neither
the support of the measure  $\mu_{(\eta,\xi)}$ nor the support of
its Fourier transform. Hence all the assumptions
in \thmref{thmA5} are satisfied and the assertion in \corref{corA6} follows.
\end{proof}

\subsection{}
Now we give a more general version of \thmref{thmA5}, in which the origin is allowed to lie in the support $\Lam$ of the measure $\mu$ or the support $S$ of its Fourier transform $\ft{\mu}$ (or both). The result can be stated as follows:

\begin{thm}
  \label{thmA8}
  Let $\mu$  be a measure satisfying 
the same assumptions as in \thmref{thmA5} except that $\Lam$ and $S$ 
may contain the origin. Then 
  \begin{equation}
    \label{eqA8.1}
    \sig := -2 a(0) \delta'_0 + \sum_{\lam \in \Lam \setminus \{0\}}
	 \frac{a(\lam)}{|\lam|}   ( \delta_{|\lam|} -  \delta_{-|\lam|})
  \end{equation}
  is a temperate distribution on $\R$ whose one-dimensional Fourier transform is
  \begin{equation}
    \label{eqA8.2}
    \ft{\sig} =  2i b(0) \alpha_k  \delta^{(k-2)}_0
    - i \sum_{s \in S \setminus \{0\}} \frac{b(s)}{|s|^{k-2}}
    \Big[ \sum_{j=0}^{(k-3)/2}
       \beta_{j,k} \, |s|^j \,
      \Big( (-1)^j \, \delta^{(j)}_{|s|} - \delta^{(j)}_{-|s|}\Big)
      \Big],
  \end{equation}
  where the coefficients $\alpha_k$ and $\beta_{j,k}$ are defined as in \eqref{eqA1.11}.
\end{thm}

If the support $\Lam$ of the measure $\mu$
does not contain the origin, then we understand 
the coefficient $a(0)$ in \eqref{eqA8.1} to be zero; and similarly,
if the origin does not belong to the support $S$ of the measure $\ft{\mu}$,
then the coefficient $b(0)$ in \eqref{eqA8.2} is zero.
If the origin lies in neither $\Lam$ nor $S$, the assertion
in \thmref{thmA8} reduces to that of \thmref{thmA5}.

 \begin{proof}[Proof of \thmref{thmA8}]
We keep using the same notation as in the proof of \thmref{thmA5}.
 We recall that if $f$ is a function  in the Schwartz  space 
$\mathcal{S}(\R)$ then there is $g \in \mathcal{S}(\R)$
such that $f(t) - f(-t) = t g(t)$. Then by the continuity of $g$ we
have $g(0) = 2 f'(0)$, and the equality $\dotprod{\sig}{f} =  
\dotprod{\mu}{G}$ used in \eqref{eqE1.1} remains valid
also in the case when $0 \in \Lam$.

To obtain the Fourier transform of $\sig$ we argue as before.
If $f$ is a Schwartz function such that $\ft{f}$ is
compactly supported, then
\begin{equation}
  \label{eqE2.1}
\dotprod{\sig}{f} = \dotprod{\mu}{G} = 
  \dotprod{\ft{\mu}}{\ft{G}} 
= b(0) \ft{G}(0) + \sum_{s \neq 0} b(s) \ft{G}(s).
\end{equation}
If we apply \remref{remD1} to the function $G$, and use
\eqref{eqE1.4} then we obtain
\[
   \ft{G}(0)  = -\frac{\alpha_k}{2\pi}\, \ft{g}^{\:(k-1)}(0)
 =  2 i \alpha_k \ft{f}^{\:(k-2)}(0).
\]
It thus follows from \eqref{eqE2.1} that
  \begin{equation}
  \label{eqE2.2}
    \dotprod{\sig}{f} =
 2 i b(0) \alpha_k \ft{f}^{\:(k-2)}(0) + 
     i \sum_{s \neq 0 } \frac{b(s)}{|s|^{k-2}}
    \Big[ \; \cdots \; \Big],
  \end{equation}
where the omitted terms in the square brackets are 
as in \eqref{eqE1.6}. This confirms \eqref{eqA8.2}.
\end{proof}

\subsection{}
Finally we  prove \thmref{thmA3} 
that extends Guinand's formula \eqref{eqA1.3} to all odd values
 of $k$ greater than $3$. Guinand's formula is
equivalent to the assertion that $\ft{\sig}_3 = -i \sig_3$,
where $\sig_k$ is the distribution 
defined in \eqref{eqA1.1} that involves
the sum-of-$k$-squares function $r_k(n)$.
\thmref{thmA3} is a consequence of the
next result which provides an expression for the Fourier transform 
$\ft{\sig}_k$ for  all odd values
 of $k$ greater than $3$.

\begin{thm}
  \label{thmE1}
  Let \(k\) be an odd integer, \(k\geq 3\). Then the Fourier transform of the temperate distribution
  \begin{equation*}
    \sigma_k := -2 \delta'_0 + \sum_{n = 1}^\infty \frac{r_{k}(n)}{\sqrt{n}}\left( \delta_{\sqrt{n}} - \delta_{-\sqrt{n}} \right)
  \end{equation*}
  is
  \begin{equation}
    \label{eqE2}
    \ft{\sig}_k =  2 i \alpha_k  \delta^{(k-2)}_0
    - i \sum_{n=1}^{\infty} \frac{r_k(n)}{n^{(k-2)/2}}
    \Big[ \sum_{j=0}^{(k-3)/2}
     \beta_{j,k} \, n^{j/2} \,
    \Big( (-1)^j \, \delta^{(j)}_{\sqrt{n}} - \delta^{(j)}_{-\sqrt{n}}\Big)
    \Big].
  \end{equation}
\end{thm}

\begin{proof}
We apply \thmref{thmA8} to the measure
$\mu = \sum_{m \in \Z^k} \delta_m$ which satisfies
$\ft{\mu} = \mu$.
\end{proof}

% =======================================

\section{Remark}

It is natural to ask what can be said about the Fourier transform of
the distribution $\sig_k$ defined in  \eqref{eqA1.1}
when $k$ is a positive \emph{even} integer
(i.e.\ $k=2,4,6,8,\dots$). We expect that for these values of $k$
the Fourier transform $\ft{\sig}_k$ does not 
have a discrete support. It would be interesting
to have a proof of this claim, but we do not address this problem 
in the present work.

% =======================================

\end{document}